\documentclass{amsart}

\usepackage{amsfonts,amssymb,amscd,amsmath,enumerate,verbatim,calc}

\newcommand{\topdeg}{\operatorname{topdeg}}
\newcommand{\LM}{\operatorname{LM}}

\newcommand{\GG}{\mathcal{G}}

\newcommand{\lt}{\operatorname{LT}}

\newcommand{\bC}{{\boldsymbol C}}

\newcommand{\BZ}{{\boldsymbol Z}}

\theoremstyle{plain}
\newtheorem{theorem}{Theorem}

\newtheorem{lemma}[theorem]{Lemma}

\newtheorem{proposition}[theorem]{Proposition}

\theoremstyle{definition}

\theoremstyle{remark}

\newtheorem{Example}[theorem]{Example}

\newcounter{hours}\newcounter{minutes}
\newcommand{\printtime}{%
        \setcounter{hours}{\time/60}%
        \setcounter{minutes}{\time-\value{hours}*60}%
        \thehours\,h\ \theminutes\,min}

\begin{document}

\title[Hilbert ideal and Coinvariants] {Gr\" {o}bner bases for the Hilbert ideal and
coinvariants of the Dihedral group $D_{2p}$}
\date{\today,\ \printtime}
\author{ Martin Kohls}
\address{Technische Universit\"at M\"unchen \\
 Zentrum Mathematik-M11\\
Boltzmannstrasse 3\\
 85748 Garching, Germany}
\email{kohls@ma.tum.de}

\author{M\"uf\.it Sezer}
\address { Department of Mathematics, Bilkent University,
 Ankara 06800 Turkey}
\email{sezer@fen.bilkent.edu.tr}
\thanks{We thank T\" {u}bitak for
funding a visit of the first author to Bilkent University, and
Gregor Kemper for funding a visit of the second author to TU
M\"unchen. Second author is also partially supported by
T\"{u}bitak-Tbag/109T384 and T\"{u}ba-Gebip/2010. }

\subjclass[2000]{13A50} \keywords{Dihedral groups, Coinvariants,
Hilbert ideal, Universal Gr\"obner Bases}
\begin{abstract}
We consider a finite dimensional representation  of the dihedral
group $D_{2p}$ over a field of characteristic two where $p$ is an
odd prime  and study the  corresponding Hilbert ideal $I_H$.  We
show that $I_H$ has a universal Gr\" {o}bner basis consisting of
invariants and monomials only. We provide sharp bounds for the
degree of an element in this basis and in a minimal generating set
for $I_H$. We also compute the top degree of coinvariants.
\end{abstract}

 \maketitle

\section{introduction}
Let $V$ be a finite dimensional representation of a finite group
$G$ over a field $F$. There is an induced action of $G$ on the
symmetric algebra $F[V]$ of $V^{*}$ that is given by $g (f)=f\circ
g^{-1}$ for $g \in G$ and $f\in F[V]$. Let $F[V]^G$ denote the
ring of invariant polynomials in $F[V]$. One of the main goals in
invariant theory is to determine $F[V]^G$ by computing the
generators and relations. A closely related object is the Hilbert
ideal, denoted $I_H$, which is the ideal in $F[V]$ generated by
invariants of positive degree. The Hilbert ideal often plays an
important role in invariant theory as it is possible to extract
information from it about the invariant ring.
 There is also substantial evidence that
the Hilbert ideal is better behaved than the full invariant ring
in terms of constructive complexity. The invariant ring is in
general not generated by invariants of degree at most the group
order when the characteristic of $F$ divides the group order (this
is known as the modular case) but it has been conjectured \cite[Conjecture
3.8.6 (b)]{MR1918599} that
 the Hilbert ideal  always is. Apart from
the non-modular case this conjecture is known to be true if $V$
is a trivial source module or if $G=\BZ_p$ and $V$ is an
indecomposable module. Furthermore,  Gr\"obner bases for $I_H$ have
been determined for some classes of groups. The reduced Gr\"obner
bases corresponding to several representations of $\BZ_p$ have been
computed in a study  of the module structure of the coinvariant
ring $F[V]_G$ which is defined to be $F[V]/I_H$, see \cite{MR2193198}.
The reduced Gr\"obner bases for the natural action of the
symmetric and the alternating group can be found in
\cite{MR1457831} and \cite{MR2246711}, respectively. These bases
have applications in coding theory, see \cite{MR1958953}.

In this paper we consider a representation of the  dihedral group
$D_{2p}$ over a field of characteristic two where $p$ is an odd
prime. Invariants of $D_{2p}$ in characteristic zero have been
studied by Schmid \cite{MR1180987} where she shows beyond other things that $\bC
[V]^{D_{2p}}$ is generated by invariants of degree at most $p+1$.
More recently, bounds for the degrees of elements in both
generating and separating sets  over an algebraically closed field
of characteristic two have been computed, see \cite{SKD2p}. We continue
further in this direction and show that the Hilbert ideal $I_H$ is
generated by invariants up to degree $p$ and not less. We also
construct a universal Gr\"obner basis for $I_H$, i.e. a set $\GG$
which forms a Gr\"obner basis of $I_H$ for any monomial order.
Somewhat unexpectedly, the only polynomials that are not invariant
in this set are monomials. Moreover, the maximal degree of a
polynomial in the basis is $p+1$. This is also atypical for
Gr\"obner basis calculations because passing from a generating set
to a Gr\"obner basis   increases  the  degrees  rapidly in
general. Then we turn our attention to the coinvariants. Of
particular interest are the top degree and the dimension of
$F[V]_G$, because a vector space basis for $F[V]_G$ yields a basis
for the invariants that can be obtained by averaging over the
group and these invariants may be crucial in efficient generation
of the whole invariant ring, see for example \cite{MR2264069}. Perhaps
among the most celebrated results on coinvariants is one due to
Steinberg \cite{MR0167535} which says that the group order $|G|$
is a lower bound for the dimension of $F[V]_G$ as a vector space,
which is sharp if and only if the invariant ring $F[V]^G$ is
polynomial, see also \cite{MR1972694}. Using the Gr\"obner basis
for $I_H$ we compute the top degree of the coinvariants of
$D_{2p}$. It turns out that for faithful representations, the top degree equals the upper bound for the
maximum degree of a polynomial in a minimal generating set that
was given in \cite{SKD2p}. Also we present upper bounds for the
top degree and the dimension of coinvariants of arbitrary finite
groups, which might be part of the folklore, but do not seem to
have appeared explicitly yet.

\section{The Hilbert ideal}\label{SetupSection}

We start by fixing our notation. Let $p\ge 3$ be an odd integer
and let $G$ denote the dihedral group of order $2p$, generated by
an element $\sigma$ of order $2$ and an element $\rho$ of order
$p$. We also let $F$ denote a field of characteristic two which contains a
primitive $p$th root of unity. We assume that $G$ acts on the polynomial ring
\[ F[V] =
F[x_1,\ldots,x_r,y_1,\ldots,y_r,z_1,\ldots,z_s,w_1,\ldots,w_s]
\]
as follows: The element $\sigma$ permutes $x_i$ and $y_i$ for
$i=1,\ldots,r$ and $z_i$ and $w_i$ for $i=1,\ldots,s$
respectively. Furthermore, $\rho$ acts trivially on $z_i$ and $w_i$
for $i=1,\ldots,s$, while $\rho(x_i)=\lambda_i x_i$ and
$\rho(y_i)=\lambda_i^{-1} y_i$ for $\lambda_i$ a non trivial
$p$-th root of unity for $i=1,\ldots,r$. Up to choice of a basis,
this is the form of an arbitrary reduced $G$-action, see
\cite{SKD2p}. We will write $u$ to denote any of the variables of
$F[V]$, and then $v$ for $\sigma(u)$. Let further $M$ denote the
subset of monomials of $F[V]$. For $m\in M^\rho$, we write $o(m)$
for the orbit sum of $m$, i.e. $o(m)=m$ if $m\in M^G$ and
$o(m)=m+\sigma(m)$ if $m\in M^\rho\setminus M^G$. Recall that
$F[V]^G$ is generated by orbit sums of $\rho$-invariant monomials.

Note that a result of Fleischmann \cite[Theorem 4.1]{MR1800251} implies that
the Hilbert ideal is generated by invariants up to degree $2p$. In
the following proposition, among other things, we sharpen this
bound to $p$.

\begin{proposition}\label{HilbertIdeal}
\begin{enumerate}
\item[(a)] The Hilbert ideal $I_H$ is generated by invariants of
positive degree at most $p$.

\item[(b)] If $m\in M^\rho$ and $u|m$, then $um\in I_H$.

\item[(c)] If $m\in M^\rho$ and $u_1$ and $u_2$ are variables such
that $u_{1}^{2}|m$ and $\rho$ acts on $u_1$ and $u_2$ by
multiplication with the same root of unity, then $mu_2\in I_H$.
\end{enumerate}
\end{proposition}

\begin{proof}
(a) Let $I$ denote the ideal of $F[V]$ generated by invariants of
positive degree at most $p$. We have to show that the Hilbert
ideal, which is generated by orbit sums of $\rho$-invariant
monomials of positive degree, equals $I$. For the sake of a proof
by contradiction, take a $\rho$-invariant monomial $m$ of minimal
degree $d$ such that $o(m)$ is not in $I$. First assume $m\in M^G$,
and take a variable $u$ appearing in $m$. Then also $v=\sigma(u)$
appears in $m$, so $uv|m$, and as $uv$ is an invariant of degree
$2$, this shows that $m\in I$. Secondly, assume $m\in
M^{\rho}\setminus M^G$. Since $d>p$, by Lemma \ref{MonDecomp} (a) we have a
factorization $m=m_1m_2$ of $m$ into two $\rho$-invariant
monomials $m_1,m_2$ of degree strictly smaller than $d$. We consider
\[
o(m)=m_1m_2+\sigma(m_1 m_2)= m_1(m_2+\sigma(m_2)) +
\sigma(m_2)(m_1+\sigma(m_1)),
\]
where $m_i+\sigma(m_i)$ for $i=1,2$ respectively are either zero
or orbit sums of $\rho$-invariant monomials of degree strictly
smaller than $d$, hence they are in $I$ by induction.

(b)  Write $m=um'$, where  $m'$ is a monomial. Then
$um=u^2m'=u(m+\sigma(m)) + uv \sigma(m')$ is in  $I_H$,
because $(m+\sigma(m))$ and $uv$ are.

(c)  Write $m=u_1^2 m'$, where  $m'$ is a monomial. Then
\[
u_2m=u_{2}u_{1}^{2}m'=u_{1} (u_2u_1m' + \sigma(u_2u_1m')) +
u_1(\sigma(u_2u_1m'))
\]
is in $I_H$: The first summand is a multiple of the
orbit sum of the $\rho$-invariant monomial $u_2u_1m'$, and the
second one is a multiple of the invariant $u_1v_1$.
\end{proof}

In the proof, we have used part (a) of the following lemma:

\begin{lemma}\label{MonDecomp}
\begin{enumerate}
\item[(a)] Every $\rho$-invariant monomial $m$ of degree at least
$p+1$ can be written as a product of two $\rho$-invariant
monomials $m_1,m_2$ whose degrees are strictly smaller than the
degree of $m$.

\item[(b)] Assume $p$ is an odd prime. The ideals
\begin{eqnarray*}
I&=&\langle \{um\,\,|\,\, m\in M^\rho \text{ and } u \text{
a variable dividing }m\}\rangle\\
\text{ and }\,\,I'&=&\langle \{um\,\,|\,\, m\in M^\rho \text{ of
degree at most }p \text{ and } u \text{ a variable dividing
}m\}\rangle
\end{eqnarray*} of $F[V]$
are equal.
\end{enumerate}

\end{lemma}

\begin{proof}
(a) In case $m$ contains a variable with trivial $\rho$-action,
the statement is obvious. Otherwise, it follows from  Proposition
\ref{SchmidCyclicGroups} applied to the characters of the
$\rho$-actions on $p+1$ arbitrary variables (counted with
multiplicity) appearing in $m$.

(b) We have to show $I\subseteq I'$, so take $um\in I$ with $u$ a
variable dividing $m$, where $m$ is a $\rho$-invariant monomial of
degree at least $p+1$. If $u$ is $\rho$-invariant, then
$um\in\langle u^2\rangle\subseteq I'$, so assume $u$ is not
$\rho$-invariant. Also we can assume that $m$ does not contain any
$\rho$-invariant variable by induction. Now Proposition
\ref{SchmidCyclicGroups} applied to $p+1$ of the  characters of
the $\rho$-action on the variables of $um$, with the character of
$u$ appearing twice, provides a $\rho$-invariant monomial $m'$
dividing $m$ of degree at most $p$, which is divisible by $u$.
Hence $um\in \langle um'\rangle\subseteq I'$.
\end{proof}

\begin{proposition}[{Schmid \cite[proof of Proposition
7.7]{MR1180987}}]\label{SchmidCyclicGroups} Let
$x_{1},\ldots,x_{t}\in ({\mathbb
  Z}/p {\mathbb Z})\setminus\{\overline{0}\}$ ($p\ge 2$ a natural number) be a sequence of
  $t\ge p+1$ nonzero elements. Then there exists a pair of
   indices
   $k_{1},k_{2}\in\{1,\ldots,t\}$, $k_{1}\ne
  k_{2}$ such that $x_{k_{1}}=x_{k_{2}}$ with the additional property that there exists a subset of indices
  $\{i_{1},\ldots,i_{r}\}\subseteq\{1,\ldots,t\}\setminus\{k_{1},k_{2}\}$ such
  that
$$
x_{k_{1}}+x_{i_{1}}+\ldots+x_{i_{r}}=\overline{0}. $$ If $p$ is
prime, any pair of
 indices
   $k_{1},k_{2}\in\{1,\ldots,t\}$, $k_{1}\ne
  k_{2}$ such that $x_{k_{1}}=x_{k_{2}}$ has this additional
  property.
\end{proposition}

Note that when $p$ is not a prime, this additional property is not
guaranteed for an arbitrary choice of indices $k_1,k_2$ with
$x_{k_1}=x_{k_2}$. For example when $p=sl$ with $s,l>1$, consider
$x_1=x_2=\bar{1}$ and $x_i=\bar{s}$ for $i=3,\ldots,p+1$ and take
$k_1=1,\, k_2=2$.\\

We recall the following notation: For a given  monomial order $<$ on $M$ and a polynomial $f$
we write $\LM (f)$ for the leading monomial of $f$. Also,  for a subset $\GG\subseteq F[V]$
and $f\in F[V]$ we write
$f\rightarrow_\GG 0$ if there exist elements $a_1,\ldots,a_n\in
F[V]$ and $g_1,\ldots,g_n\in \GG$ such that
$f=a_1g_1+\ldots+a_ng_n$ and $\LM(f)\ge \LM(a_ig_i)$ for
$i=1,\ldots,n$. In this case we say $f$ reduces to zero modulo
$\GG$. Notice that $f\rightarrow_\GG 0$ implies $af\rightarrow_\GG
0$ for any $a\in F[V]$.

\begin{lemma}\label{GGLemma}
Let $f,g\in F[V]$ with $\LM(f)>\LM(g)$. Then $f\rightarrow_\GG 0$
and $g\rightarrow_\GG 0$ for a set  $\GG\subseteq F[V]$ imply
$(f+g)\rightarrow_\GG 0$.
\end{lemma}
\begin{proof}
We have $f=\sum a_ig_i$ and $g=\sum b_ig_i$ for some $a_i,b_i\in
F[V]$ and $g_i\in \GG$ with $\LM (a_ig_i)\le \LM (f)$ and $\LM
(b_ig_i)\le \LM(g)< \LM (f)$. Then $(f+g)=\sum (a_i+b_i)g_i$ gives
$(f+g)\rightarrow_\GG 0$ because $\LM ((a_i+b_i)g_i)\le \max \{\LM
(a_ig_i), \LM (b_ig_i) \}\le \LM (f)=\LM(f+g)$.
\end{proof}

From now on, we will assume that $p$ is an odd prime.\\

Let $\GG$ denote the following set of polynomials:
\[
\begin{array}{rl}
m+\sigma(m)&\text{ for } m\in M^{\rho}\setminus M^G \text{ of
degree at most }p,\\
um &\text{ for } m\in M^{\rho}  \text{ of degree at most } p
\text{ and } u \text{  a variable dividing }m,\\
x_iy_i, \;  z_jw_j &\text{ for } i=1,\ldots,r \text { and } j=1,\ldots,s.\\
\end{array}\]
We show that $\GG$ is a universal Gr\"obner basis of $I_H$.
We need the following lemma.

\begin{lemma}
Let $m\in M^{\rho}$. Then $(m+\sigma(m))\rightarrow_\GG 0$.
\end{lemma}

\begin{proof}
We assume $m\in M^{\rho}\setminus M^G$ since $m+\sigma(m)=0$ if
$m\in M^G$. We also take $\deg(m)> p$ because otherwise
$m+\sigma(m)\in \GG$. Then by  Lemma \ref{MonDecomp} (a) there
exist $\rho$-invariant monomials $m_1,m_2$ of degree strictly
smaller than the degree of $m$ such that $m=m_1m_2$. Without loss
of generality, we assume $m>\sigma(m)$.  So we have either $m_1 >
\sigma (m_1)$ or $m_2
>\sigma (m_2)$. We harmlessly assume $m_1 > \sigma (m_1)$.
Consider the equation $$m+\sigma(m)=m_1m_2+\sigma(m_1m_2) = m_2
(m_1 + \sigma(m_1)) + \sigma(m_1)(m_2+\sigma(m_2)).$$ By induction
on the degree both $m_1 + \sigma(m_1)$ and $m_2 + \sigma(m_2)$
reduce to zero modulo $\GG$ and hence, so do their respective
monomial multiples $m_2 (m_1 + \sigma(m_1))$ and
$\sigma(m_1)(m_2+\sigma(m_2))$. Hence the result follows from the
previous lemma because we have $\LM (m_2 (m_1 +
\sigma(m_1)))=m_1m_2$ and $m_1m_2>\sigma(m_1)m_2$ and
$m_1m_2>\sigma(m_1)\sigma (m_2)$.
\end{proof}
\begin{theorem}
For $p$ an odd prime, $\GG$ forms a universal Gr\"obner basis of
$I_H$.
\end{theorem}

\begin{proof}
First note that by the second assertion of Proposition
\ref{HilbertIdeal} all elements of $\GG$ lie in $I_H$. Conversely,
by the first assertion of Proposition \ref{HilbertIdeal}, $I_H$ is
generated by orbit sums $o(m)$ of monomials $m\in M^{\rho}$ of
degree at most $p$. If $m\not\in M^G$, then $o(m)=m+\sigma(m)\in
\GG$, by construction. Otherwise, if $u|m$, we have $uv|m$, so
again $o(m)=m\in \GG$. This establishes that the ideal generated
by $\GG$ is exactly $I_H$.

Next we show that the polynomials in $\GG$ satisfy  Buchberger's
criterion. Recall that for $f_1,f_2\in F[V]$, the $s$-polynomial
$s(f_1,f_2)$ is defined to be $\frac{T}{\lt (f_1)}f_1-\frac{T}{\lt
(f_2)}f_2$, where $T$ is the least common multiple of the leading
monomials of $f_1$ and $f_2$ and $\lt (f)$ denotes the lead term
of the polynomial $f$. Buchberger's criterion says that $\GG$ is a
Gr\"obner Basis of $I_H$ if and only if $s(f_1,f_2)\rightarrow_\GG
0$ for all
 $f_1,f_2\in \GG$. Since the $s$-polynomial of two monomials is zero, we just check the $s$-polynomials of $m+\sigma(m)$ for $ m\in M^{\rho}\setminus M^G $ with
  each of the four families of polynomials in $\GG$.
  We will also use the well known fact that $s(f_1,f_2)$ reduces to zero
  modulo $\{f_1,f_2\}$ if the
  leading monomials of $f_1$ and $f_2$ are relatively prime, see \cite[Exercise 9.3]{MR2766370}.

 1)  Let $m=u_1^{a_1}\cdots u_k^{a_k}m'$ and  $n=u_1^{b_1}\cdots u_k^{b_k}n'$ be
  monomials in  $M^{\rho}\setminus M^G$ of degree at most $p$ with $a_j,b_j>0$ for $1\le j\le k$ and  $m'$ and $n'$ are relatively prime monomials.
  We further assume that neither $m'$ nor $n'$ is divisible by any
  of $u_j$ for $1\le j\le k$ and $m>\sigma (m)$ and $n>\sigma (n)$.
  Let $f_1,f_2$ denote $m+\sigma (m)$ and  $n+\sigma (n)$, respectively. Notice that
  $s(f_1,f_2)=\frac{T}{\lt (f_1)}(\sigma (m))-\frac{T}{\lt (f_2)}(\sigma (n))$. If $a_j>b_j$ for
  some $1\le j\le k$, then  $\frac{T}{\lt (f_2)}$ is divisible by $u_j$
  and so $\frac{T}{\lt (f_2)}(\sigma (n))$ is divisible by $u_jv_j$ because $\sigma (n)$ is
  divisible by $v_j$.  Similarly, if
   $b_{j'}>a_{j'}$ for some $1\le j'\le k$, then $\frac{T}{\lt (f_1)}(\sigma (m))$ is divisible by $u_{j'}v_{j'}$.
   It follows that if there are indices $1\le j,j'\le k$ such that $a_j>b_j$ and  $b_{j'}>a_{j'}$, then $s(f_1,f_2)\rightarrow_\GG
0$.  So we may
   assume $a_j\ge b_j$
   for $1\le j\le k$. Therefore
    we are reduced to two cases.

   First assume that $a_j\ge b_j$ for $1\le j\le k$ and for one of the indices the  inequality is strict, say
   $a_1>b_1$. As  in the previous paragraph $\frac{T}{\lt (f_2)}(\sigma (n))$ is divisible by $u_1v_1$. Meanwhile,
   we have $\frac{T}{\lt (f_1)}(\sigma (m))=n'v_1^{a_1}\cdots v_k^{a_k}\sigma (m')$. But since $n$ is in $M^{\rho}$,
   $\rho$ acts on $n'$ and on $v_1^{b_1}\cdots v_k^{b_k}$ by multiplication with the same root of unity.
   So $n'v_1^{a_1-b_1}\cdots v_k^{a_k-b_k}\sigma (m')$ is in $M^{\rho}$ as well because it is obtained by
   multiplying the $\rho$-invariant monomial  $v_1^{a_1}\cdots v_k^{a_k}\sigma (m')$ with
    $\frac{n'}{v_1^{b_1}\cdots v_k^{b_k}}.$ Since $a_1>b_1>0$, this shows that $\frac{T}{\lt (f_1)}(\sigma (m))$ is
    divisible by the product of the $\rho$-invariant monomial $n'v_1^{a_1-b_1}\cdots v_k^{a_k-b_k}\sigma (m')$  and the variable $v_1$ that divides this monomial.
 By Lemma \ref{MonDecomp} (b), $\frac{T}{\lt (f_1)}(\sigma
(m))$ is also divisible by a monomial in $\GG$.

Secondly, assume that $a_j=b_j$ for $1\le j\le k$.  Then we get
$s(f_1,f_2)=v_1^{a_1}\cdots v_k^{a_k}(n'\sigma(m')+m'\sigma
(n'))$. But $\rho$ multiplies $m'$ and $n'$ with the same root of
unity and hence it multiplies  $n'$ and $\sigma (m')$ with
reciprocal roots of unity. This puts  $n'\sigma(m')$ (and
$m'\sigma (n')$) in $M^{\rho}$. Hence
$s(f_1,f_2)\rightarrow_{\GG}0$, by the previous lemma.

2)   We compute the $s$-polynomial $s(f_1,f_2)$, where
$f_1=m+\sigma (m)$ for a monomial $m$ in $M^{\rho}$ of degree at
most $p$ and $f_2$ is product of a $\rho$-invariant monomial of
degree at most $p$ with a variable that divides this monomial. As before, we assume $m>\sigma (m)$. Write $m=u_1^{a_1}\cdots
u_k^{a_k}m'$ and $f_2=u_1^{b_1}\cdots u_k^{b_k}n'$ where
$a_j,b_j>0$ with relatively prime monomials $m'$ and $n'$. We
further assume $m'$ and $n'$ are not divisible by any of $u_j$. We
have $s(f_1,f_2)=\frac{T}{\lt (f_1)}(\sigma (m))$. Notice that if
$b_j>a_j$ for some $1\le j\le k$, then $\frac{T}{\lt (f_1)}$ is
divisible by $u_j$ and so $\frac{T}{\lt (f_1)}(\sigma (m))$ is
divisible by $u_jv_j$. Hence $s(f_1,f_2)$ reduces to zero modulo
$\GG$. Therefore we assume $a_j\ge b_j$ for $1\le j\le k$. So,
$s(f_1,f_2)=n'v_1^{a_1}\cdots v_k^{a_k}\sigma (m')$. By
construction there is a variable $w$ such that $w^2$ divides $f_2$
and $f_2/w$ is in $M^{\rho}$. We consider two cases.

First assume that $w^2$ divides $n'$. We have
\[
s(f_1,f_2)=n'v_1^{a_1}\cdots v_k^{a_k}\sigma (m')=\left(\frac{n'\sigma (m')v_1^{a_1-b_1}\cdots v_k^{a_k-b_k}}{w}\right)(wv_1^{b_1}\cdots v_k^{b_k
}).
\]
Since $f_2/w$ is in $M^\rho$, $\rho$ multiplies $n'/w$ and
$v_1^{b_1}\cdots v_k^{b_k}$ with the same (non-zero) scalar.
Therefore, since $\sigma (m')v_1^{a_1}\cdots v_k^{a_k}\in
M^{\rho}$, we get $\frac{n'\sigma (m')v_1^{a_1-b_1}\cdots
v_k^{a_k-b_k}}{w}\in M^{\rho}$ as well. Hence $s(f_1,f_2)$ is
divisible by the product of $w$ with a $\rho$-invariant monomial
that is divisible by $w$. By Lemma \ref{MonDecomp} (b),
$s(f_1,f_2)$ is divisible by  a monomial in $\GG$.

 Since $n'$ and $u_1^{b_1}\cdots u_k^{b_k}$ are relatively prime, we can assume as the remaining case   that $w$
 does not divide $n'$. Then
$w=u_j$ for some $1\le j\le k$. Say, $w=u_1$. We also have $a_1\ge
b_1\ge 2$. Similar to the first case we have
$$s(f_1,f_2)=n'v_1^{a_1}\cdots v_k^{a_k}\sigma (m')=(n'\sigma (m')v_1^{a_1-b_1+1}v_2^{a_2-b_2}\cdots v_k^{a_k-b_k})(v_1^{b_1-1}v_2^{b_2}\cdots v_k^{b_k}).$$
Notice that since $f_2/u_1\in M^{\rho}$,  $\rho$ acts on $n'$ and
$v_1^{b_1-1}v_2^{b_2}\cdots v_k^{b_k}$ by multiplication with the
same scalar.  Hence $(n'\sigma
(m')v_1^{a_1-b_1+1}v_2^{a_2-b_2}\cdots v_k^{a_k-b_k})$ lies in
$M^{\rho}$ because $\sigma (m')v_1^{a_1}\cdots v_k^{a_k}$ is
already $\rho$-invariant. It follows that, since $a_1-b_1+1\ge 1$
and $b_1-1\ge 1$, $s(f_1,f_2)$ is  divisible by the product of
$v_1$ with a $\rho$-invariant monomial that is divisible by $v_1$.
So we get that $s(f_1,f_2)$ is divisible by  a monomial in $\GG$
by Lemma \ref{MonDecomp} (b).

3) We compute the $s$-polynomial $s(f_1,f_2)$ where $f_1=m+\sigma
(m)$ ($m>\sigma (m)$) for a monomial $m$ in $M^{\rho}$ of degree at most $p$ and
$f_2$ is a product $uv$ for some variable $u$. Since we assume $m$
and $uv$ are not relatively prime we take $m=u^{a}m'$
 where $u$ does not divide $m'$.  If $v$  divides $m'$
 then both $m$ and $\sigma (m)$ are divisible by $uv$ and so $s(f_1,f_2)$
 equals $\sigma(m)$. Hence it is divisible by $uv$ and we are done.
 Therefore we assume $v$ does not divide $m$ so we have $s(f_1,f_2)=v\sigma
 (m)$. But $v$ divides $\sigma (m)$, and the latter is in $M^{\rho}$ and is of degree
 at most $p$. Hence $v\sigma (m)$ is an element of $\GG$.
\end{proof}

\section{Bounds for coinvariants}
 Before we
specialize to the dihedral group, we start this section with a
general result that is probably part of the folklore, but it seems
it has not been written down explicitly yet. In the following
theorem, $G$ is an arbitrary finite group and $F$ an arbitrary
field. If the field is large enough,  Dades' algorithm
\cite[Proposition 3.3.2]{DerksenKemper}
 provides a homogeneous system of parameters with each
element of degree $|G|$. Note that field extensions do not affect
the degree structure of coinvariants, so in particular we can
assume $d_i=|G|$ for $i=1,\ldots,n$ in the following theorem.

\begin{theorem}\label{GenCov}
Assume $d_1,\ldots,d_n$ are the degrees of a homogeneous system of
parameters of $F[V]^G$. Then we have
\[\begin{array}{rrrcl}
(a)&  &\topdeg (F[V]_G) &\le& \sum_{i=1}^{n} (d_i-1),\\

(b)&& \dim (F[V]_G)&\le& \prod_{i=1}^{n} d_i.\\
\end{array}
\]
In particular, we have $\topdeg (F[V]_G) \le \dim(V)(|G|-1)$ and
$\dim (F[V]_G) \le |G|^n$. If the system of parameters generates
$F[V]^G$, we have equalities in (a) and (b).
\end{theorem}

\begin{proof}
Let $A$ be the subalgebra of $F[V]^G$ generated by a homogeneous
system of parameters with the given degrees. As the group $G$ is
finite and $K[V]$ is Cohen-Macaulay, we have that $K[V]$ is a free
$A$-module, say $K[V]=\bigoplus_{i=1}^{r}Ag_i$ with
$g_1,\ldots,g_r$ homogeneous elements of degrees $m_1\le\ldots\le
m_r$. Then $r$ equals the dimension and $m_r$ equals the top
degree of $F[V]/(A_{+}\cdot F[V])$, respectively. As $A_+\subseteq
F[V]_{+}^{G}$, the numbers $r$ and $m_r$ are bigger than or equal
to the dimension and top degree of $F[V]/I_H$ respectively. As the
Hilbert series of $F[V]/(A_{+}\cdot F[V])$ is given by
\[H(t)=\frac{\prod_{i=1}^{n}(1-t^{d_i})}{(1-t)^n}=\prod_{i=1}^{n}
(1+t+t^2+\ldots+t^{d_i-1}),\] we get $m_r=\deg
H(t)=\sum_{i=1}^n(d_i-1)$ and $r=H(1)=\prod_{i=1}^n d_i$, which
proves (a) and (b).
\end{proof}

Now we restrict ourselves to the coinvariants of the dihedral groups.

\begin{theorem}
For $p$ an odd prime, the top degree of the coinvariants of the
dihedral group $D_{2p}$ in characteristic two equals $s+\max(r,p)$
if $r \ge 1$, and equals $s$ if $r=0$.
\end{theorem}

\begin{proof}
We write $d$ for the top degree of $F[V]_G$. For a polynomial
$f\in F[V]$, let $\deg_{xy}f$ denote the degree of $f$ in the
variables $x_1,\ldots,x_r,y_1,\ldots,y_r$, and define $\deg_{zw}
f$ similarly. Let $m$ be a monomial. The proof consists of four
observations. (i) If $\deg_{zw}m>s$, then  $m$ is divisible either
by $z_iw_i$ or one of $z_i^2$ or $w_i^2$ for some $i=1,\ldots,s$,
in particular $m\in I_H$. This implies $d\le s$ in case $r=0$.
(ii) If $\deg_{xy}m > \max(r,p)$ then $\deg_{xy}m> r$ implies that
$m$ is divisible by $x_iy_i$ or $x_i^2$ or $y_i^2$ for some
$i=1,\ldots,r$. In the first case $m\in I_H$, so without loss of
generality we can assume $x_{i}^{2}|m$ for some $i$. By
Proposition \ref{SchmidCyclicGroups}, $\deg_{xy}m>p$ implies that
there exists a factorization $m=(x_{i}n)x_in'$ such that $x_in$ is
a $\rho$-invariant monomial of degree at most $p$. As $x_{i}^{2}n$
is an element of $\GG$, we have $m\in I_H$. Now (i) and (ii) imply
that if $\deg(m)>s+\max(r,p)$, then $m\in I_H$, hence $d\le
s+\max(r,p)$. (iii) We claim that $n:=y_1\cdots y_r w_1\cdots w_s$
is not in $I_H$, hence $d\ge r+s$. Otherwise, $n$ would be
divisible by the leading monomial of an element of $\GG$.   Since no variable in $n$ has
multiplicity bigger than one, $n$ is in fact divisible  by
$\LM(m+\sigma(m))$ for some monomial  $m\in M^\rho\setminus M^G$
of degree at most $p$. As $\GG$ is a universal Gr\"obner basis, we
can choose a lexicographic order $>$ with $x_i>y_j$ and $z_i>w_j$
for all $i,j$ and  assume $m>\sigma(m)$. We fix this order until
the end of the proof. Then $m|n$ implies that $m=y_{i_1}\cdots
y_{i_k} w_{j_1}\cdots w_{j_{l}}$, but then
$\sigma(m)=x_{i_1}\cdots x_{i_k} z_{j_1}\cdots z_{j_l}>m$ by the
choice of our order, a contradiction. (iv) Finally if $r\ge 1$, we
claim that $n:=y_1^p w_1\cdots w_s$ is not in $I_H$, hence $d\ge
p+s$. As before, $n\in I_H$ would imply that $n$ is divisible by the leading
monomial of an element of $\GG$. Notice that a $\rho$-invariant monomial divisor
of $n$ either is divisible by $y_1^p$ or is not divisible by $y_1$ at
all. It follows that the only leading monomial of a member of $\GG$ that divides $n$ is
of the form $\LM(m+\sigma(m))$ for some monomial $m\in
M^\rho\setminus M^G$ of degree at most $p$. Assuming
$m>\sigma(m)$, we see that $m$ would be of the form $w_{i_1}\cdots
 w_{i_k}$ or $y_1^p w_{i_1}\cdots w_{i_k}$, so $\sigma(m)$ would
 be of the form $z_{i_1}\cdots
 z_{i_k}$ or $x_1^p z_{i_1}\cdots z_{i_k}$ respectively. In each
 case, we have the contradiction $\sigma(m)>m$ by choice of our
 monomial order.
\end{proof}
\begin{Example}
We take $r=1,\,\,s=0$ and write $x$ and $y$ for $x_1$ and $y_1$.
Then $F[V]^G=F[xy,x^p+y^p]$, see e.g. \cite[Remark 5]{SKD2p}. In
particular, all elements in the Hilbert ideal of degree less than
$p$ are divisible by $xy$, so the bound in Proposition
\ref{HilbertIdeal} (a) is sharp. A universal Gr\"obner Basis of
$I_H$ is given by $\GG=\{xy,x^p+y^p,x^{p+1},y^{p+1}\}$. If we
choose lexicographic order with $x>y$, we see that the lead term
ideal of $I_H$ is minimally spanned by $\{xy,x^p,y^{p+1}\}$. In
particular, any Gr\"obner Basis must contain an element of degree
$p+1$. The generators of $F[V]^G$ form a homogeneous system of
parameters in degrees $d_1=2$ and $d_2=p$. Thus, Theorem
\ref{GenCov} yields the sharp bounds $\topdeg(F[V]_G)\le
(d_1-1)+(d_2-1)=p=s+\max(r,p)$ and $\dim(F[V]_G)\le d_1 d_2=2p$.
\end{Example}

Note that in case $r\ge 1$, the top degree of the coinvariants is
the same as the upper bound for the degrees of elements in a
minimal  generating set for the invariant ring that is given in
\cite[Theorem 4]{SKD2p}. If $r=0$, what we really consider are the
vector invariants of the permutation action of $\BZ_{2}$. In this
case, the fact that the top degree  of the coinvariants is $s$
also follows from \cite[Theorem 2.1]{MR2193198}. The maximal
degree of elements in a minimal generating set in this case is
also given by $s$ if $s\ge 2$, see \cite{MR1051222}. It would
hence be tempting to conjecture that the invariant ring is always
generated by invariants of degree at most the top degree of the
coinvariants. However, in case $r=0$ and $s=1$, we have
$F[z,w]^{G}=F[zw,z+w]$, but the top degree of the coinvariants is
one.

\bibliographystyle{plain}
\bibliography{Bibliography_Version_11}
 \end{document}